\documentclass{amsart}
\usepackage{amsmath}
\usepackage{amssymb}
\usepackage{amscd}
\usepackage{color}
\usepackage[all]{xy}

\newcommand{\Hom}{\operatorname{Hom}}
\newcommand{\im}{\operatorname{im}}
\newcommand{\supp}{\operatorname{supp}}

\newtheorem{theorem}{Theorem}[section]
\newtheorem{lemma}[theorem]{Lemma}
\newtheorem{corollary}[theorem]{Corollary}
\newtheorem{proposition}[theorem]{Proposition}
\newtheorem{definition}[theorem]{Definition}
\newtheorem{example}[theorem]{Example}
\newtheorem{remark}[theorem]{Remark}

\begin{document}

\title[On monoid graded semihereditary rings]
{On monoid graded semihereditary rings}
\author[H. F. G. Al-Kharsan, P. Sahandi \\and N. Shirmohammadi]
{Haneen Falah Ghalib Al-Kharsan, Parviz Sahandi \\and Nematollah Shirmohammadi}

\address{(Al-Kharsan) Department of Pure Mathematics, Faculty of Mathematics, Statistics and Computer Science,  University of Tabriz, Tabriz, Iran.} \email{almwswyhnyn90@gmail.com}
\address{(Sahandi) Department of Pure Mathematics, Faculty of Mathematics, Statistics and Computer Science,  University of Tabriz, Tabriz, Iran.} \email{sahandi@ipm.ir, sahandi@tabrizu.ac.ir}
\address{(Shirmohammadi) Department of Pure Mathematics, Faculty of Mathematics, Statistics and Computer Science,  University of Tabriz, Tabriz, Iran.} \email{ne.shirmohammadi@gmail.com, shirmohammadi@tabrizu.ac.ir}



\begin{abstract}
In order to study graded left hereditary and left semihereditary rings graded by a cancelation monoid in terms of their modules, we need to revisit graded free, projective, injective, and flat modules and provide graded versions of specific results concerning these modules, like Baer's criterion on injectivity and Lazard's theorem on flatness. Then, among other things, we can give some characterizations of graded left hereditary and left semihereditary rings, in particular, of graded-Pr\"{u}fer and graded-Dedekind domains.
\end{abstract}

\maketitle

\section{Introduction}

After Kaplansky's work on Dedekind and Pr\"{u}fer domains \cite{ka52}, Cartan and Eilenberg defined the notion of hereditary and semihereditary rings in \cite{ce56}. Let $R$ be an associative ring with nonzero identity. Recall that $R$ is \emph{left hereditary} if every left ideal is a projective $R$-module. It is known that $R$ is left hereditary if and only if each submodule of a projective left $R$-module is projective if and only if each quotient of an injective left $R$-module is injective \cite[Page 14]{ce56}. A commutative hereditary integral domain is a Dedekind domain. It was proved that $R$ is a Dedekind domain if and only if every divisible $R$-module is injective \cite[Theorem 4.24]{ro09}.

Recall also that $R$ is \emph{left semihereditary} if every finitely generated left ideal is a projective $R$-module. It is known that $R$ is left semihereditary if and only if $R$ is left coherent, and each submodule of a flat left $R$-module is flat if and only if each torsion-free right $R$-module is flat \cite[Theorem 4.32]{ro09}. A commutative semihereditary integral domain is a Pr\"{u}fer domain. It is shown that $R$ is a Pr\"{u}fer domain if and only if every finitely generated torsion-free $R$-module is projective \cite[Theorem 4.34]{ro09}.

The graded analogue of Dedekind and Pr\"{u}fer domains was defined and studied by several authors (see \cite{no82, ac13, acz16, s16}). This paper aims to define and study the graded analogues of left hereditary and semihereditary rings. Let $\Gamma$ be a cancelation monoid, and $R=\bigoplus_{\alpha \in \Gamma}R_{\alpha}$ be a $\Gamma$-graded ring. To this purpose, one needs to study $\Gamma$-graded free, projective, injective, and flat modules. These modules were extensively studied in case where $\Gamma=\mathbb{Z}$ in \cite{ff74, bh98}, and in case where $\Gamma$ is a group in \cite{no04}. The first two of them were investigated in case where $\Gamma$ is a cancelation monoid in \cite{li12}. Then, in this paper, we examine $\Gamma$-graded injective and flat modules where $\Gamma$ is a cancelation monoid.

The organization of this paper is as follows. In the remainder of the introduction and Section 2, we recall some preliminaries on graded rings and modules. In Section 3, after reviewing some results on the graded free and projective modules, we investigate the graded divisible, injective, and flat modules. We provide a graded analogue of Baer's Theorem and show that every graded injective module is graded divisible and every $R$-module admits a graded injective envelope. Then, we define and study graded flat modules and extend Lazard's Theorem to the monoid-graded case. More precisely, it is shown that any graded flat module is a direct limit of finitely generated graded free modules. In Section 4, we study graded left hereditary rings. A $\Gamma$-graded ring $R=\bigoplus_{\alpha \in \Gamma}R_{\alpha}$ is called \emph{graded left hereditary} if every left homogeneous ideal is projective. Then, among other things, we prove the graded version of the Cartan-Eilenberg Theorem: $R$ is graded left hereditary if and only if every graded submodule of a graded projective $R$-module is graded projective if and only if every quotient of a graded injective $R$-module is graded injective. Then we show that a commutative graded integral domain is graded-Dedekind if and only if every graded-divisible left $R$-module is graded-injective. Finally, in Section 5, we study graded left semihereditary rings. A $\Gamma$-graded ring $R$ is called \emph{graded left semihereditary} if every finitely generated left homogeneous ideal is projective. Among other things, it is shown that $R$ is graded left semihereditary if and only if every finitely generated graded submodule of a graded projective left $R$-module is graded projective if and only if $R$ is graded left coherent and every homogeneous left ideal is graded flat. It is also shown that a commutative graded integral domain $R$ is a graded-Pr\"{u}fer domain if and only if every finitely generated torsion-free graded $R$-module is graded projective.

\subsection{Graded Rings and modules}

Let $\Gamma$ be a cancelation monoid, that is, $\Gamma$ is a semigroup with a neutral element $\varepsilon$, not necessarily commutative, and
$\Gamma$ satisfies the (left and right) cancelation law: $\alpha\alpha_1=\alpha\alpha_2$ implies $\alpha_1=\alpha_2$, and $\alpha_1\alpha=\alpha_2\alpha$ implies $\alpha_1=\alpha_2$ for all $\alpha,\alpha_1,\alpha_2\in\Gamma$. By a \emph{$\Gamma$-graded ring} $R$, we mean an associative ring with the multiplicative identity 1, which has a direct sum decomposition $R=\bigoplus_{\alpha \in \Gamma}R_{\alpha}$, where each $R_{\alpha}$ is an additive subgroup of $R$ such that $R_{\alpha}R_{\beta}\subseteq R_{\alpha\beta}$ for all $\alpha,\beta\in\Gamma$. For each $\alpha\in\Gamma$, we call $0\neq a\in R_{\alpha}$ a homogeneous element of degree $\alpha$ and write $\deg(a)=\alpha$. For the basics about graded rings and
graded modules we used in this paper, one may refer to \cite{no04}, though it mainly deals with group $G$-graded rings.

Let $R=\bigoplus_{\alpha \in \Gamma}R_{\alpha}$ be a $\Gamma$-graded ring as above. By the definition, it is clear that the degree-$\varepsilon$ part $R_{\varepsilon}$ is a subring of $R$, and by using the cancelation law on $\Gamma$, it is easy to check that $1\in R_{\varepsilon}$. Let $I$ be a left (resp. right) ideal of $R$. If $I$ is generated by homogeneous elements, then we say that $I$ is a \emph{homogeneous left (resp. right) ideal} of $R$. By using the cancelation law on $\Gamma$, it is straightforward to check that $I$ is a homogeneous left (resp. right) ideal of $R$ if and only if $I=\bigoplus_{\alpha \in \Gamma}(I\cap R_{\alpha})$ if and only if $a=\Sigma_{i}a_{\alpha_i}\in I$ with $a_{\alpha_i}\in R_{\alpha_i}$ implies $a_{\alpha_i}\in I$ for all $i$ if and only if $R/I=\bigoplus_{\alpha \in \Gamma}(R_{\alpha}+I)/I$. If each graded left ideal of $R$ is finitely generated, then $R$ is called a \emph{left gr-Noetherian ring}.

Let $R=\bigoplus_{\alpha \in \Gamma}R_{\alpha}$ be a $\Gamma$-graded ring. By a $\Gamma$-graded left $R$-module $M$, we mean a left $R$-module which has a direct sum
decomposition $M=\bigoplus_{\alpha \in \Gamma}M_{\alpha}$, where each $M_{\alpha}$ is an additive subgroup of $M$, such that $R_{\alpha}M_{\beta}\subseteq M_{\alpha\beta}$ for all $\alpha,\beta\in\Gamma$. Similarly, one can define a $\Gamma$-graded right $R$-module. For each $\alpha\in\Gamma$, we call $0\neq x\in M_{\alpha}$ a \emph{homogeneous element} of degree $\alpha$ and write $\deg(x)=\alpha$. By \cite[Lemma 1.1]{li12}, if $\{x_i\}_{i\in J}$ is a homogeneous generating set of $M$ with $\deg(x_i)=\alpha_i$, $i\in J$, then
$$M_{\alpha}=\sum_{i\in J\text{, }\alpha'_i\alpha_i=\alpha}R_{\alpha'_i}x_i.$$
Let $N$ be a submodule of $M$. If $N$ is generated by some homogeneous elements, then we say that \emph{$N$ is a graded submodule of $M$}. By using the cancelation law on $\Gamma$, it is straightforward to check that $N$ is a graded submodule of $M$ if and only if $N=\bigoplus_{\alpha \in \Gamma}(N\cap M_{\alpha})$ if and only if $x=\Sigma_{i}x_{\alpha_i}\in N$ with $x_{\alpha_i}\in M_{\alpha_i}$ implies $x_{\alpha_i}\in N$ for all $i$ if and only if $M/N=\bigoplus_{\alpha \in \Gamma}(M_{\alpha}+N)/N$. We let $H(M)=\bigcup_{\alpha \in \Gamma}M_{\alpha}$ be the set of homogeneous elements of $M$.

Let $M=\bigoplus_{\alpha \in \Gamma}M_{\alpha}$ and $N=\bigoplus_{\alpha \in \Gamma}N_{\alpha}$ be $\Gamma$-graded left $R$-modules. If $\varphi:M\to N$ is an $R$-module homomorphism such that $\varphi(M_{\alpha})\subseteq N_{\alpha}$ for all $\alpha\in\Gamma$, then we call $\varphi$ a \emph{homogeneous homomorphism}. In the case that $\varphi:M\to N$ is a homogeneous homomorphism, it is easy to see that $\ker(\varphi)$ and $\im(\varphi)$ are graded left submodules of $M$ and $N$, respectively.

\section{Preliminaries}

Let $\Gamma$ be a cancelation monoid with neutral element $\varepsilon$, and $R=\bigoplus_{\alpha \in \Gamma}R_{\alpha}$ a $\Gamma$-graded ring. In this section, we recall some preliminaries of $\Gamma$-graded $R$-modules.

Let $\{M_i,\varphi_{ij}\}_{i,j\in I}$ be a direct system of $\Gamma$-graded left $R$-modules, i.e., $I$ is a directed partially ordered set and, for $i\leq j$, $\varphi_{ij}:M_i\to M_j$ is a homogeneous homomorphism  which is compatible with the ordering. Then $M=\varinjlim_{i\in I}M_i$ is a $\Gamma$-graded left $R$-module
with homogeneous components $M_{\alpha}=\varinjlim_{i\in I}M_{i\alpha}$. For a detailed construction of
such direct limits see \cite[II, Section 11.3, Remark 3]{b98}. In particular, let $\{M_i\mid i\in I\}$ be $\Gamma$-graded left $R$-modules. Then $\bigoplus_{i\in I}M_i$ has a natural $\Gamma$-graded left $R$-module given by $(\bigoplus_{i\in I}M_i)_{\alpha}=\bigoplus_{i\in I}M_{i\alpha}$, for all $\alpha\in\Gamma$.

Now, let $\{N_i,\psi_{ij}\}_{i,j\in I}$ be another direct system of $\Gamma$-graded left $R$-modules and set $N=\varinjlim_{i\in I}N_i$. We say that $\{\tau_i\}:\{M_i,\varphi_{ij}\}\to \{N_i,\psi_{ij}\}$ is a \emph{morphism of direct systems} if $\tau_i:M_i\to N_i$, $i\in I$, is a homogeneous homomorphism, such that the following diagram is commutative for all $i\leq j$:
\begin{displaymath}
\xymatrix{ M_i \ar[r]^{\tau_i} \ar[d]_{\varphi_{ij}} & N_i \ar[d]^{\psi_{ij}} \\
M_j \ar[r]_{\tau_j} & N_j. }
\end{displaymath}
In this case, there exists a homomorphism $\tau:M\to N$ such that the diagram
\begin{displaymath}
\xymatrix{ M_i \ar[r]^{\tau_i} \ar[d]_{\varphi_i} & N_i \ar[d]^{\psi_i} \\
M \ar[r]_{\tau} & N, }
\end{displaymath}
commutes for all $i\in I$, where $\varphi_i, \psi_i$ are natural homogeneous homomorphisms. It can be seen that $\tau$ is homogeneous.

Let $M$, $N$ be $\Gamma$-graded left $R$-modules and $\alpha\in\Gamma$. A \emph{homogeneous $R$-module homomorphism of degree $\alpha$} is an $R$-module homomorphism $f:M\to N$, such that $f(M_{\beta})\subseteq N_{\beta\alpha}$ for any $\beta\in\Gamma$. Let $\Hom_R(M,N)$ be the group of all $R$-module homomorphism from $M$ to $N$, and denote by $\Hom_{\alpha}(M,N)$ the subgroup of $\Hom_R(M,N)$ consisting of all homogeneous $R$-module homomorphisms of degree $\alpha$. By using the cancelation law on $\Gamma$, it is easy to see that $\Hom_{\alpha}(M,N)$, $\alpha\in\Gamma$, constitutes a direct sum. We put $$\sideset{^*}{_{R}}{\Hom}(M,N):=\bigoplus_{\alpha\in\Gamma}\Hom_{\alpha}(M,N).$$ Again, by using the cancelation law on $\Gamma$, one notices that the functors $\sideset{^*}{_{R}}{\Hom}(-,N)$ and $\sideset{^*}{_{R}}{\Hom}(M,-)$ are left exact on the category of $\Gamma$-graded left $R$-modules. In general, $\sideset{^*}{_{R}}{\Hom}(M,N)$ may not equal to $\Hom_R(M,N)$, but under certain conditions we have $\sideset{^*}{_{R}}{\Hom}(M,N)=\Hom_{R}(M,N)$, for instance, for the case that $\Gamma$ is a group, if $M$ is finitely generated (\cite[Corollary 2.4.4]{no04}), or if both $M$ and $N$ have finite support (\cite[Corollary 2.4.5]{no04}), in particular in the case where the grading group $\Gamma$ is finite.

In a similar manner, for $\Gamma$-graded rings $R$ and $S$, we can consider the $\Gamma$-graded $R-S$-bimodule $N$. That is, $N$ is a $R-S$-bimodule and additionally $N=\bigoplus_{\alpha\in\Gamma}N_{\alpha}$ is a $\Gamma$-graded left $R$-module and a $\Gamma$-graded right $S$-module, i.e., $R_{\alpha}N_{\gamma}S_{\beta}\subseteq N_{\alpha\gamma\beta}$, where $\alpha,\gamma,\beta\in\Gamma$. Assume now that $N$ is a graded $R-S$-bimodule, then it is clear that $\sideset{^*}{_{R}}{\Hom}(M,N)$ has a natural structure of a $\Gamma$-graded right $S$-module, with homogeneous components $\Hom_{\alpha}(M,N)$, $\alpha\in\Gamma$.

Let $M$ be a $\Gamma$-graded right $R$-module and $N$ be a $\Gamma$-graded left $R$-module. We will observe that the tensor product $M\otimes_RN$ has a natural $\Gamma$-graded $\mathbb{Z}$-module structure (here $\mathbb{Z}$ has trivial $\Gamma$-graded structure). Since each of $M_{\alpha}$, $\alpha\in\Gamma$, is a right $R_{\varepsilon}$-module and similarly $N_{\alpha}$, $\alpha\in\Gamma$, is a left $R_{\varepsilon}$-module, then $M\otimes_{R_{\varepsilon}}N$ can be decomposed as a direct sum $$M\otimes_{R_{\varepsilon}}N=\bigoplus_{\alpha\in\Gamma}(M\otimes N)_{\alpha},$$ where $$(M\otimes N)_{\alpha}=\{\sum_im_i\otimes n_i\mid m_i\in H(M), n_i\in H(N), \deg(m_i)+\deg(n_i)=\alpha\}.$$ It now follows that $M\otimes_RN=(M\otimes_{R_{\varepsilon}}N)/J$, where J is a subgroup of $M\otimes_{R_{\varepsilon}}N$ generated by the homogeneous elements $$\{mr\otimes n-m\otimes rn\mid m\in H(M), n\in H(N), r\in H(R)\}.$$ This shows that $M\otimes_RN$ is also a $\Gamma$-graded $\mathbb{Z}$-module.

\begin{lemma}\label{lim}
Let $I$ be a directed partially ordered set, and $\{M_i,\varphi_{ij}\}$, $\{N_i,\psi_{ij}\}$, and $\{L_i,\theta_{ij}\}$ be direct systems of $\Gamma$-graded left $R$-modules.
\begin{enumerate}
  \item Assume that there are morphisms of direct systems over $I$ $$\{M_i,\varphi_{ij}\}\to\{N_i,\psi_{ij}\}\to\{L_i,\theta_{ij}\}$$
such that the sequence $M_i\to N_i\to L_i$ is exact for each $i\in I$. Then there is an exact sequence of $\Gamma$-graded left $R$-modules with homogeneous homomorphisms
$$\underset{i\in I}{\varinjlim}M_i\to\underset{i\in I}{\varinjlim}N_i\to\underset{i\in I}{\varinjlim}L_i.$$
  \item Assume that $A$ is a $\Gamma$-graded right $R$-module. Then there is a homogeneous isomorphism $$A\otimes_R(\underset{i\in I}{\varinjlim}M_i)\cong \underset{i\in I}{\varinjlim}(A\otimes_RM_i).$$
  \item Assume that $J$ is a cofinal subset in $I$. Then $J$ is directed and there is a homogeneous isomorphism $\underset{j\in J}{\varinjlim}M_j\cong\underset{i\in I}{\varinjlim}M_i$.
\end{enumerate}
\end{lemma}
\begin{proof}
For (1), it remains to prove the exactness, which in turn follows by \cite[Proposition 8.9]{so00}, and (2) follows as in the ungraded case \cite[Corollary 8.8]{so00}, while (3) is an easy exercise.
\end{proof}

\section{Graded projective, injective, and flat modules}

Let $\Gamma$ be a cancelation monoid with neutral element $\varepsilon$ and $R=\bigoplus_{\alpha \in \Gamma}R_{\alpha}$ a $\Gamma$-graded ring. In this section, we investigate graded projective, injective, and flat modules, and we compare each of them with the usual ungraded one.

\subsection{Graded-projective modules}

Recall that a $\Gamma$-graded left $R$-module $F$ is called \emph{graded free} (for short, \emph{gr-free}) if it has an $R$-basis consisting of homogeneous elements. One notices that there is a graded and free modules that is not gr-free \cite[Page 21]{no04}; hence
gr-free is a stronger property than ``graded plus free''. Given any free $R$-module $F=\bigoplus_{i\in J}Re_i$ and an arbitrarily chosen family $\{\alpha_i\}_{i\in J}$ of elements of $\Gamma$, one can give a $\Gamma$-graduation on $F$ such that $\{e_i\}_{i\in J}$ will be a homogeneous basis with $\deg(e_i)=\alpha_i$. This shows that if $M$ is a $\Gamma$-graded left $R$-module with homogeneous generating set $\{x_i\}_{i\in J}$, and the gr-free $R$-module $F=\bigoplus_{i\in J}Re_i$, with $\deg(e_i)=\deg(x_i)$, then the map $\varphi:F\to M$ defined by $\varphi(e_i)=x_i$ is a homogeneous epimorphism. Hence, every nonzero $\Gamma$-graded left $R$-module is a graded homomorphic image of some gr-free $R$-module (see \cite[Page 2698]{li12} for details).

As in \cite{li12}, a graded left $R$-module $P$ is called a \emph{graded projective module} (for short a \emph{gr-projective module}) if for any diagram of $\Gamma$-graded left $R$-modules and homogeneous homomorphisms
\begin{displaymath}
\xymatrix{  &
P \ar[d]^{f} \ar@{-->}[ld]_{h} \\
M \ar[r]_{g} & N \ar[r] &0. }
\end{displaymath}
there is a homogeneous homomorphism $h:P\to M$ with $gh=f$.

\begin{proposition}\label{proj}
Assume that $\Gamma$ is a cancelation monoid, $R=\bigoplus_{\alpha \in \Gamma}R_{\alpha}$ is a $\Gamma$-graded ring and $P$ is a $\Gamma$-graded left $R$-module. Then the following are equivalent:
\begin{enumerate}
  \item $P$ is projective.
  \item $P$ is gr-projective.
  \item $\sideset{^*}{_{R}}{\Hom}(P,-)$ is an exact functor in the category of $\Gamma$-graded left $R$-modules.
  \item Every short exact sequence of $\Gamma$-graded left $R$-module and homogeneous homomorphisms $0\to L\stackrel{f}\to M\stackrel{g}\to P\to 0$ splits via a homogeneous homomorphism.
  \item $P$ is a graded direct summand of a gr-free $\Gamma$-graded left $R$-module.
\end{enumerate}
\end{proposition}
\begin{proof}
For $(1)\Leftrightarrow(2)\Leftrightarrow(5)$ see \cite[Proposition 1.2]{li12}.

$(2)\Rightarrow(3)$ As we have already seen, $\sideset{^*}{_{R}}{\Hom}(P,-)$ is left exact. The right exactness follows immediately from the definition of
gr-projective modules.

$(3)\Rightarrow(4)$ Applying the exact functor $\sideset{^*}{_{R}}{\Hom}(P,-)$ to the short exact sequence $0\to L\stackrel{f}\to M\stackrel{g}\to P\to 0$ gives rise to a homogeneous homomorphism $h:P\to M$ such that $gh=1_P$.

$(4)\Rightarrow(5)$ As we recalled above, there exists a homogeneous homomorphism $\varphi:F\to P$, where $F$ is a gr-free $\Gamma$-graded left $R$-module. By the assumption, there is a homogeneous homomorphism $h:P\to F$ such that $gh=1_P$. Then $\theta:P\oplus\ker(g)\to F$ given by $\theta(x,y)=h(x)+y$ is a homogeneous isomorphism.
\end{proof}

Assume that $M$ is a $\Gamma$-graded left $R$-module which is finitely presented in the category of left $R$-modules. Let $F_1\to F_0\stackrel{f}{\to} M\to 0$ be a finite presentation of $M$ (in the usual sense); so that $0\to\ker(f)\to F_0\stackrel{f}{\to} M\to0$ is exact. Note that $M$ is finitely generated; so we can choose a finite set of generators $x_1,\ldots,x_n$, such that each $x_i$ is a homogeneous element of degree $\alpha_i$. Let $F'_0:=\oplus_{i=1}^nRe_i$ be the gr-free $\Gamma$-graded left $R$-module with homogeneous basis $\{e_1,\ldots,e_n\}$, $\deg(e_i)=\alpha_i$. Define the homogeneous homomorphism $f':F'_0\to M$ by $e_i\mapsto x_i$. Hence, there is an exact sequence $0\to\ker(f')\to F'_0\stackrel{f'}{\to} M\to0$. It follows from Schanuel's Lemma \cite[Proposition 3.12]{ro09} that $\ker(f)\oplus F'_0\cong\ker(f')\oplus F_0$. Since $\ker(f)$ is finitely generated, we see that $\ker(f')$ is finitely generated, too. Hence $\ker(f')$ is an image of a finitely generated gr-free left $R$-module $F'_1$. Therefore, $M$ is finitely presented in the category of $\Gamma$-graded left $R$-modules.

\subsection{Graded-injective modules}

The notion of graded injective modules was investigated in \cite{ff74, bh98} in the case where $\Gamma=\mathbb{Z}$ and in \cite{no04} in the case where $\Gamma$ is a group. It is known that the notions of graded injective and (usual) injective are not the same. For example, over an arbitrary field $k$, consider the Laurent polynomial ring $R=k[X,X^{-1}]$ with $\mathbb{Z}$-gradation. Then by \cite[Remark 2.3.3]{no04}, the left $R$-module $R$ is graded injective but not injective.

A $\Gamma$-graded left $R$-module $E$ is said to be a \emph{graded injective module} (for short, \emph{gr-injective module}) if for every homogeneous homomorphisms $f:M\to E$ and $g:M\to N$ with $g$ a homogeneous injection, there exists a homogeneous homomorphism $h:N\to E$ such that the diagram
\begin{displaymath}
\xymatrix{
0 \ar[r] & M \ar[r]^{g}\ar[d]_{f} & N \ar@{-->}[dl]^{h}  \\
& E &
 }
\end{displaymath}
commutes.

We now provide some useful characterizations of gr-injective modules, which will be used in the sequel (see \cite[Corollary 2.4.8]{no04} when $\Gamma$ is a group).

\begin{theorem}\label{baer} (The graded version of Baer's Theorem) Let $\Gamma$ be a cancelation monoid and $M$ be a $\Gamma$-graded left $R$-module. Then $M$ is gr-injective if and only if for any homogeneous left ideal $I$ of $R$, any $\alpha\in\Gamma$ and any homogeneous homomorphism $f:I\to M$ of degree $\alpha$, there exists $m_{\alpha}\in M_{\alpha}$ such that $f(r)=rm_{\alpha}$ for any $r\in I$.
\end{theorem}
\begin{proof}
Assume that $M$ is gr-injective, $I$ is a homogeneous left ideal of $R$, and $f:I\to M$ is a homogeneous homomorphism of degree $\alpha\in\Gamma$. There is a gr-free left $R$-module $F=Re$ with homogeneous basis $\{e\}$ of degree $\alpha$ by \cite[Page 2698]{li12}. Consider the following diagram
\begin{displaymath}
\xymatrix{
0\ar[r]  & Ie \ar@{^{(}->}[r]\ar[d]_{\overline{f}} & F\\
& M, &}
\end{displaymath}
where $\overline{f}(re)=f(r)$, which is a homogeneous homomorphism. By assumption there is a homogeneous homomorphism $\overline{g}:F\to M$ such that $\overline{g}\mid_{Ie}=\overline{f}$. Putting $m_{\alpha}:=\overline{g}(e)$, we have $f(r)=\overline{f}(re)=\overline{g}(re)=rm_{\alpha}$.

Conversely, assume we have the following diagram with exact row of $\Gamma$-graded left $R$-modules and homogeneous homomorphisms:
\begin{displaymath}
\xymatrix{
0\ar[r]  & A \ar[r]^{\varphi}\ar[d]_{f} & B\\
& M. &
}
\end{displaymath}
Let $\Sigma$ be the set of all ordered pairs $(B',g')$ where $\varphi(A)\subseteq B'\subseteq B$, $B'$ is a graded submodule of $B$ and
\begin{displaymath}
\xymatrix{
0\ar[r]  & A \ar[r]^{\varphi}\ar[d]_{f} & B'\ar[dl]^{g'}\\
& M &
}
\end{displaymath}
commutes. Partially order $\Sigma$ by $(B'',g'')\leq(B',g')$ when $B''\subseteq B'$ and $g'\mid_{B''}=g''$. Then $\Sigma$ is nonempty and inductive. Let $(B',g')$ be a maximal element of $\Sigma$. We claim that $B'=B$. Otherwise, choose a homogeneous element $x\in B\setminus B'$ of degree $\alpha$. Let $I=\{r\in R\mid rx\in B'\}$ which is a homogeneous left ideal of $R$. Define $\overline{f}:I\to M$ by $\overline{f}(r)=g'(rx)$ for $r\in I$. It is clear that $\overline{f}$ is a homogeneous homomorphism of degree $\alpha$. By assumption, there is $m_{\alpha}\in M_{\alpha}$ such that $\overline{f}(r)=rm_{\alpha}$ for $r\in I$. Set $B'':=B'+Rx$ and define $g'':B''\to M$ by $g''(b+rx)=g'(b)+rm_{\alpha}$. It can be seen that $g''$ is a well defined homomorphism. We show that $g''$ is homogeneous. Indeed, assume that $\gamma\in\Gamma$ and $b''\in B''_{\gamma}$. Then there are $b\in B'_{\gamma}$ and $r\in R_{\beta}$ with $\beta\alpha=\gamma$ such that $b''=b+rx$. Hence $rm_{\alpha}$ is homogeneous of degree $\beta\alpha=\gamma$; so $g''(b'')=g'(b)+rm_{\alpha}$ is homogeneous of degree $\gamma$. This means that $(B'',g'')$ properly extends $(B',g')$, contradicting the maximality of $(B',g')$.
\end{proof}

\begin{proposition}\label{inj}
Assume that $\Gamma$ is a cancelation monoid, $R=\bigoplus_{\alpha \in \Gamma}R_{\alpha}$ is a $\Gamma$-graded ring and $E$ is a $\Gamma$-graded left $R$-module. Then the following are equivalent:
\begin{enumerate}
  \item $E$ is gr-injective.
  \item $\sideset{^*}{_{R}}{\Hom}(-,E)$ is an exact functor in the category of $\Gamma$-graded left $R$-modules.
  \item Every short exact sequence of $\Gamma$-graded left $R$-modules and homogeneous homomorphisms $0\to E\to M\to N\to 0$ splits via a homogeneous homomorphism.
\end{enumerate}
\end{proposition}
\begin{proof}
$(1)\Rightarrow(2)$ and $(2)\Rightarrow(3)$ are easy and $(3)\Rightarrow(1)$ is straightforward by pushout construction.
\end{proof}

It is well known that every (not necessarily graded) injective $R$-module is divisible \cite[Corollary (3.17)$^\prime$]{la98}, and a commutative integral domain $R$ is a Dedekind domain if and only if every divisible $R$-module is injective \cite[Theorem 4.24]{ro09}. The notion of a graded divisible module was introduced in \cite[Page 47]{rh16} for group-graded modules. This notion can be generalized to monoid-graded modules as in the following. Recall that a nonzero element $a\in R$ is a \emph{right zero divisor} if there exists a nonzero element $b\in R$ such that $ba=0$. The element $a\in R$ is called to be a \emph{right nonzero divisor} if it is not a right zero divisor.

\begin{definition}
Assume that $\Gamma$ is a cancelation monoid. A $\Gamma$-graded left $R$-module $M$ is called graded divisible (for short, gr-divisible) if for $\alpha,\beta,\gamma\in\Gamma$ with $\alpha\gamma=\beta$ and every right nonzero divisor $r\in R_{\alpha}$ and every $y\in M_{\beta}$, there exists $x\in M_{\gamma}$ such that $rx=y$.
\end{definition}

It is clear that every graded and divisible $R$-module is gr-divisible. But the converse is no longer true. For example, over an arbitrary field
$k$, consider the Laurent polynomial ring $R=k[X,X^{-1}]$ with $\mathbb{Z}$-gradation. Then by \cite[Remark 2.3.3]{no04}, the left $R$-module $R$ is gr-injective but not injective. Hence, it is gr-divisible by the following proposition but not divisible since $R$ is a Dedekind domain (see \cite[Theorem 4.24]{ro09}).

\begin{proposition}\label{injdiv}
Assume that $\Gamma$ is a cancelation monoid. Any gr-injective left $R$-module is gr-divisible.
\end{proposition}
\begin{proof}
Assume that $E$ is a gr-injective left $R$-module, and choose $\alpha,\beta,\gamma\in\Gamma$ with $\alpha\gamma=\beta$, a right nonzero divisor $r\in R_{\alpha}$ and $y\in M_{\beta}$. It follows that the map $\varphi:Ra\to E$, defined by $\varphi(ra)=ry$ is well defined, and using the cancelation law on $\Gamma$, it is a homogeneous homomorphism of degree $\gamma$. Then by Theorem \ref{baer}, there exists $x\in E_{\gamma}$ such that $\varphi(r)=rx$ for $r\in Ra$. Hence, $y=\varphi(a)=ax$.
\end{proof}

In the following, we show that every $\Gamma$-graded left $R$-module has a graded injective envelope. For this, we follow \cite[Sections 3.2 and 3.6]{bh98}.

Let $M$ be a $\Gamma$-graded left $R$-module, and let $N$ be a graded submodule of $M$. The extension $N\subseteq M$ is called a \emph{graded-essential extension} (for short, \emph{gr-essential extension}) if $U\cap N\neq 0$ for any graded submodule $0\neq U\subseteq M$. If in addition $N\neq M$, the extension is called a \emph{proper gr-essential extension}. As in the case where $\Gamma=\mathbb{Z}$ (see \cite[Lemma 1.1]{ff74}) or $\Gamma$ is a group (see \cite[Proposition 2.3.5]{no04}), one can show that $N\subseteq M$ is a gr-essential extension if and only if it is an essential extension.

\begin{proposition}\label{ess}
A $\Gamma$-graded left $R$-module $N$ is gr-injective if and only if it has no proper gr-essential extension.
\end{proposition}
\begin{proof}
Assume that $M$ is a $\Gamma$-graded left $R$-module such that $N$ is a graded submodule of $M$. If $N$ is gr-injective, then there exists a graded submodule $W$ of $M$ such that $M=N\oplus W$ by Proposition \ref{inj}. Then $N\cap W=0$ and so, if the extension is gr-essential, $W=0$. It follows that $N=M$.

Conversely, suppose that $N$ has no proper gr-essential extension. Given a homogeneous monomorphism $\varphi:U\to V$ and a homogeneous homomorphism $f:U\to M$, we want to construct a homogeneous homomorphism $h:V\to N$ such that $f=h\varphi$. We consider the pushout diagram
\begin{displaymath}
\xymatrix{ U \ar[r]^{\varphi} \ar[d]_{f} &
V \ar[d]^{g} \\
N \ar[r]^{\psi} & W. }
\end{displaymath}
Indeed, we may choose $W:=(V\oplus N)/C$ with $C:=\{(\varphi(x),-f(x))\mid x\in U\}$ which is a graded submodule of $V\oplus N$; $g$ and $\psi$ are the natural homogeneous projections. Here $\psi$ is a homogeneous monomorphism, since $\varphi$ is a monomorphism. Thus we may consider $N$ as a graded submodule of $W$. It now follows from Zorn's lemma that there exists a maximal graded submodule $D\subseteq W$ such that $N\cap D=0$, and so $N$ may even be considered as a graded submodule of $W/D$; obviously, $W/D$ is a gr-essential extension of $N$. It follows that $N=W/D$, since $N$ has no proper gr-essential extension, and so $W=N\oplus D$. Let $\pi:W\to N$ be the natural projection of $W$ onto the first summand. The composition $\pi g:V\to N$ is an extension of $f$.
\end{proof}

We now prove that any $\Gamma$-graded left $R$-module has a graded injective envelope. In analogy to the definition in the ungraded case, a $\Gamma$-graded left $R$-module $E$ is called a \emph{gr-injective envelope of $M$} if it is gr-injective and a gr-essential extension of $M$.

\begin{theorem}\label{injectivehull}
Any $\Gamma$-graded left $R$-module $M$ admits a gr-injective envelope, and any two gr-injective envelopes are isomorphic as $\Gamma$-graded left $R$-modules.
\end{theorem}
\begin{proof}
We embed $M$ into a (not necessarily graded) injective $R$-module $T$. We consider the set $\mathcal{S}:=\{N\mid M\subseteq N\text{ is gr-essential and } N\text{ is a submodule of } T\}$. We define a partial order $\leq$ on $\mathcal{S}$ by setting $N_1\leq N_2$ if $N_1$ is a graded submodule of $N_2$. By Zorn's lemma, $\mathcal{S}$ has a maximal element $E$; hence $M\subseteq E$ is gr-essential and $E$ is a submodule of $T$. Suppose $E$ is not gr-injective; then $E$ has a proper gr-essential extension $E\subsetneq E'$ by Proposition \ref{ess}. As $T$ is injective, there exists an $R$-module homomorphism $\varphi:E'\to T$ (not necessarily homogeneous), extending the inclusion $E\subseteq T$. We claim that $\varphi$ is injective. Indeed, assume that there is a nonzero element $x\in \ker\varphi$, so that $x=x_{\alpha_1}+\cdots+x_{\alpha_n}$ is the decomposition of $x$ into nonzero homogeneous components. We show by induction on $n$ that there exists a homogeneous element $a\in R$ such that $0\neq ax\in E$. Since $\varphi\mid_E$ is injective, this gives a contradiction.

If $n=1$, $x$ is homogeneous, and the assertion follows since the extension $E\subseteq E'$ is gr-essential. Now suppose that $n>1$. We choose a homogeneous element $a\in R$ such that $0\neq ax_{\alpha_1}\in E$. Let $x':=x-x_{\alpha_1}=x_{\alpha_2}+\cdots+x_{\alpha_n}$. If $ax'=0$ then $0\neq ax=ax_{\alpha_1}\in E$ and we are done. Otherwise $ax'\neq0$, and by our induction hypothesis we may choose a homogeneous element $b\in R$ such that $0\neq bax'\in E$. Then $0\neq bax= bax'+bax_{\alpha_1}\in E$.

Next let $\widetilde{E}:=\im\varphi$. As $\varphi$ is injective, we may give $\widetilde{E}$ a natural graded structure $\widetilde{E}_{\alpha}=\varphi(E'_{\alpha})$ for all $\alpha\in \Gamma$. Then $E\subsetneq \widetilde{E}$ is a proper gr-essential extension with $\widetilde{E}\subseteq T$, contradicting the maximality of $E$. The uniqueness of the gr-injective envelope is easily obtained.
\end{proof}

\subsection{Graded-flat modules}

Lazard's theorem says that any flat module is a direct limit of finitely generated free modules. In this subsection, we deal with graded flat modules and extend Lazard's theorem to monoid-graded modules.

Let $\Gamma$ be a cancelation monoid and $R$ be a $\Gamma$-graded ring. A $\Gamma$-graded left $R$-module $F$ is said to be a \emph{graded-flat module} (for short, \emph{gr-flat module}) if $f\otimes_R1_F:M\otimes_RF\to N\otimes_RF$ is a homogeneous monomorphism for every homogeneous monomorphism $f:M\to N$ of $\Gamma$-graded right $R$-modules. It is clear that if $F$ is a $\Gamma$-graded $R$-module which is flat, then it is a gr-flat module.

We begin with the following simple lemma.

\begin{lemma}\label{flat}
Every direct limit of a directed system over a directed set $I$ of gr-flat left $R$-modules is gr-flat.
\end{lemma}
\begin{proof}
Assume that $\{F_i,\varphi_{ij}\}$ is a direct system over $I$ of gr-flat left $R$-modules and that $f:M\to N$ is a homogeneous monomorphism of graded right $R$-modules. Then, we obtain a commutative diagram with homogeneous homomorphisms
\begin{displaymath}
\xymatrix{ M\otimes_R(\underset{i\in I}{\varinjlim}F_{i}) \ar[r]^{f\otimes 1} \ar[d] & N\otimes_R(\underset{i\in I}{\varinjlim}F_{i}) \ar[d] \\
\underset{i\in I}{\varinjlim}(M\otimes_RF_{i}) \ar[r]_{\vartheta} & \underset{i\in I}{\varinjlim}(N\otimes_RF_{i}), }
\end{displaymath}
where the vertical maps are the homogeneous isomorphisms, and $\vartheta$ is the homogeneous monomorphism by Lemma \ref{lim}. Hence $f\otimes 1$ is a monomorphism. Therefore $\underset{i\in I}{\varinjlim}F_{i}$ is gr-flat.
\end{proof}

Assume that $C$ is a $\Gamma$-graded left $R$-module and $D$ is a graded submodule of $C$. Let $\mathcal{A}$ and $\mathcal{B}$ be nonempty families of graded submodules of $C$. As in \cite[Section 8.4]{so00}, by these data one can define a \emph{graded $(C,D)$-subquotient system} subject to the following conditions:
\begin{enumerate}
  \item $\mathcal{A}$ and $\mathcal{B}$ are directed under inclusions.
  \item $C=\bigcup_{A\in\mathcal{A}}A$ and $D=\bigcup_{B\in\mathcal{B}}B$.
  \item For each $B\in\mathcal{B}$, there exists $A\in\mathcal{A}$ with $A\supseteq B$.
\end{enumerate}
Given a graded $(C,D)$-subquotient system, set $$\mathcal{I}=\{(A,B)\in\mathcal{A}\times\mathcal{B}\mid A\supseteq B\}.$$
Partially order $\mathcal{I}$ by $(A,B)\leq(A',B')$ if and only if $A\subseteq A'$ and $B\subseteq B'$.
For $i=(A,B)\in \mathcal{I}$ we set $A=A_i$ and $B=B_i$.

For $i, j\in\mathcal{I}$ with $i\leq j$, define homomorphisms $\varphi_{ij}:A_i/B_i\to A_j/B_j$ by $\varphi_{ij}(x+B_i)=x+B_j$ and $\psi_i:A_i/B_i\to C/D$ by $\psi_{i}(x+B_i)=x+D$. It is clear that $\varphi_{ij}$ and $\psi_i$ are homogeneous and $\{A_i/B_i\}_{i\in\mathcal{I}}$ together with $\varphi_{ij}$ is a direct system of graded modules. The proof of the following proposition is the same as in the ungraded case \cite[Proposition 8.12]{so00}, so we omit it.

\begin{proposition}\label{8.12}
Assume $\mathcal{A}$ and $\mathcal{B}$ give a graded $(C,D)$-subquotient system. Then, as in the above notation, $\mathcal{I}$ is directed and there is a homogeneous isomorphism of graded modules $\underset{i\in \mathcal{I}}{\varinjlim}A_i/B_i\cong C/D$.
\end{proposition}

\begin{lemma}\label{8.13}
Suppose that $A$ is a $\Gamma$-graded left $R$-module, with $A=A_1\oplus A_2$, where $A_1$ and $A_2$ are graded submodules of $A$. Suppose $B_1$ is a graded submodule of $A_1$, and $B_2$ is a graded submodule of $A_2$. Suppose $\theta:A_1/B_1\to A_2/B_2$ is a homogeneous homomorphism, and suppose $A_1$ is generated by homogeneous elements $x_1,\ldots,x_n$. For each $k=1,\ldots,n$, choose the homogeneous element $y_k\in A_2$ such that $\theta(x_k+B_1)=y_k+B_2$. Let $B$ be the (graded) submodule of $A$ generated by $B_1, B_2$, and the homogeneous elements $x_1-y_1,\ldots,x_n-y_n$. Let $\psi_i:A_i/B_i\to A/B$ be the natural homogeneous homomorphism $\psi_{i}(x+B_i)=x+B$. Then $\psi_2$ is a homogeneous isomorphism, and $\psi_2\theta=\psi_1$.
\end{lemma}
\begin{proof}
Since $\psi_2$ is homogeneous, it remains to show that it is an isomorphism, which in turn follows from \cite[Lemma 8.13]{so00} that $\psi_2$ is an isomorphism and $\psi_2\theta=\psi_1$.
\end{proof}

\begin{proposition}\label{8.14}
Assume that $\Gamma$ is a cancelation monoid, $R=\bigoplus_{\alpha \in \Gamma}R_{\alpha}$ is a $\Gamma$-graded ring and $E$ is a $\Gamma$-graded left $R$-module. With the above notation, there exists a graded $(C,D)$-subquotient system with the following properties:
\begin{enumerate}
  \item $C/D\cong E$ in the category of $\Gamma$-graded left $R$-modules.
  \item $A_i/B_i$ is $\Gamma$-graded finitely presented for all $i\in \mathcal{I}$.
  \item If $F$ is $\Gamma$-graded finitely presented, and if $\eta:F\to C/D$ is a homogeneous homomorphism, then there exist an $i\in \mathcal{I}$ and a homogeneous isomorphism $\eta':F\to A_i/B_i$ such that the triangle
      \begin{displaymath}
\xymatrix{
A_i/B_i \ar[rd]^{\psi_i} &   \\
  & C/D\\
 F \ar[uu]^{\eta'}\ar[ru]_{\eta} & }
\end{displaymath}
      commutes.
  \item If $i\in \mathcal{I}$ and $\psi_i=\rho\sigma$, where $\sigma:A_i/B_i\to F$ and $\rho:F\to C/D$ are homogeneous homomorphisms, with $F$ a $\Gamma$-graded finitely presented module, then there exist $j\in \mathcal{I}$, $j\geq i$, and a homogeneous isomorphism $\tau:F\to A_j/B_j$ such that the diagram
      \begin{displaymath}
\xymatrix{
& A_j/B_j \ar[rd]^{\psi_j} &   \\
A_i/B_i\ar[ru]^{\varphi_{ij}} \ar[rd]_{\sigma}& & C/D\\
& F \ar[uu]^{\tau}\ar[ru]_{\eta} & }
\end{displaymath}
      commutes.
\end{enumerate}
\end{proposition}
\begin{proof}
Let $C$ be the gr-free $\Gamma$-graded left $R$-module with homogeneous basis $\{c_{(e,n)}\mid e\in H(E), n\in \mathbb{N}\}$ such that $\deg(c_{(e,n)})=\deg(e)$ for each $e\in H(E), n\in \mathbb{N}$. Let $\pi:C\to E$ be the homogeneous epimorphism defined by $\pi(c_{(e,n)})=e$. Hence $D:=\ker\pi$ is a graded submodule of $C$. Let $\mathcal{A}$ be the family of graded submodules of $C$ generated by finite subsets of $\{c_{(e,n)}\mid e\in H(E), n\in \mathbb{N}\}$ (hence finitely generated gr-free), and let $\mathcal{B}$ be the family of finitely generated graded submodules of $D$. With this setup we have a graded $(C,D)$-subquotient system and $C/D\cong E$ as $\Gamma$-graded $R$-modules, while each $A_i/B_i$ is graded finitely presented. Hence (1) and (2) are true.

Part (3) is a special case of (4), and the proof of (4) follows as in the ungraded case; see the proof of part (d) in \cite[Proposition 8.14]{so00} using Lemma \ref{8.13} instead of \cite[Lemma 8.13]{so00}.
\end{proof}

Assume that $F$ and $E$ are $\Gamma$-graded left $R$-modules. It is easy to see that there exists a natural group homomorphism $$\Phi_{F,E}:\sideset{^*}{_{R}}{\Hom}(F,R)\otimes_RE\to \sideset{^*}{_{R}}{\Hom}(F,E),$$ defined by $\Phi_{F,E}(f\otimes e)=(\Phi_{F,E})_{f\otimes e}$ in which $(\Phi_{F,E})_{f\otimes e}(x)=f(x)\cdot e$, for $f\in\sideset{^*}{_{R}}{\Hom}(F,R)$ and $e\in E$, such that $$\Phi_{F,E}((\sideset{^*}{_{R}}{\Hom}(F,R)\otimes_RE)_{\alpha})\subseteq \Hom_{\alpha}(F,E),$$ for each $\alpha\in\Gamma$.
It is easy to see that $\Phi_{-,E}:\sideset{^*}{_{R}}{\Hom}(-,R)\otimes_RE\to \sideset{^*}{_{R}}{\Hom}(-,E)$ is a natural transformation.

\begin{theorem}(Graded version of Lazard's Theorem)\label{lazard}
Assume that $\Gamma$ is a cancelation monoid, $R=\bigoplus_{\alpha \in \Gamma}R_{\alpha}$ is a $\Gamma$-graded ring and $E$ is a $\Gamma$-graded left $R$-module. The following are equivalent:
\begin{enumerate}
  \item $E$ is gr-flat.
  \item For all graded finitely presented left $R$-module $F$, $\Phi_{F,E}$ is isomorphic.
  \item For all graded finitely presented left $R$-module $F$, $\Phi_{F,E}$ is epimorphic.
  \item $E$ is a direct limit of finitely generated gr-free $R$-modules.
\end{enumerate}
\end{theorem}
\begin{proof}
(1)$\Rightarrow$(2) is obtained using the naturality of $\Phi_{-,E}$ as the ungraded case (see \cite[Proposition 4.18]{so00}), and (2)$\Rightarrow$(3) is trivial.

(3)$\Rightarrow$(4) Let $\mathcal{I}$, $A_i$, $B_i$, $\psi_i$, $\varphi_{ij}$ be as in Proposition \ref{8.14}, and set $$\mathcal{J}:=\{j\in \mathcal{I}\mid A_j/B_j\text{ is gr-free}\}.$$ Assume that $i\in\mathcal{I}$. Then $\psi_i\in\Hom_{\varepsilon}(A_i/B_i,C/D)$ lies in the image of $\Phi_{A_i/B_i,C/D}$. Hence there are $\sigma_k\in\Hom_{\alpha_k}(A_i/B_i,R)$, $k=1,\ldots,n$, and homogeneous elements $\overline{y}_1,\ldots,\overline{y}_n\in C/D$, with $\deg(\overline{y}_k)=\beta_k$, $\alpha_k\beta_k=\varepsilon$, such that $\Phi_{A_i/B_i,C/D}(\Sigma_{k=1}^n\sigma_k\otimes \overline{y}_k)=\psi_i$. Then $\psi_i(\overline{x})=\Sigma_{k=1}^n\sigma_k(\overline{x})\overline{y}_k$. Define homogeneous homomorphisms $\sigma=(\sigma_1,\ldots,\sigma_n):A_i/B_i\to L$, where $L:=\bigoplus_{k=1}^nRe_k$ is gr-free with $\deg(e_k)=\beta_k$, and $\rho:L\to C/D$ by $\rho(r_1,\ldots,r_n)=\Sigma_{k=1}^nr_k\overline{y}_k$. Thus $\psi_i=\rho\sigma$, so by Proposition \ref{8.14}, there exist $j\in \mathcal{I}$, $j\geq i$, and a homogeneous isomorphism $\tau:L\to A_j/B_j$, for which the diagram
\begin{displaymath}
\xymatrix{
& A_j/B_j \ar[rd]^{\psi_j} &   \\
A_i/B_i\ar[ru]^{\varphi_{ij}} \ar[rd]_{\sigma}& & C/D\\
& L \ar[uu]^{\tau}_{\cong}\ar[ru]_{\eta} & }
\end{displaymath}
is commutative. So, $A_j/B_j\cong L$ is a gr-free $\Gamma$-graded $R$-module, hence $j\in \mathcal{J}$. This means that $\mathcal{J}$ is cofinal in $\mathcal{I}$. Hence by Lemma \ref{lim}(3) and Proposition \ref{8.12}, $\mathcal{J}$ is directed, and $$E\cong C/D\cong\underset{i\in \mathcal{I}}{\varinjlim}A_i/B_i\cong\underset{j\in \mathcal{J}}{\varinjlim}A_j/B_j.$$ (4)$\Rightarrow$(1) follows from Lemma \ref{flat}.
\end{proof}

The following corollary generalizes \cite[Proposition 3.2]{ff74} in the case where $\Gamma=\mathbb{Z}$ and \cite[Exercise 11, Page 72]{no04} in the case where $\Gamma$ is a group.

\begin{corollary}\label{lazard2}
Assume that $\Gamma$ is a cancelation monoid, and that $R=\bigoplus_{\alpha \in \Gamma}R_{\alpha}$ is a $\Gamma$-graded ring. A $\Gamma$-graded left $R$-module $F$ is gr-flat if and only if it is a flat $R$-module.
\end{corollary}
\begin{proof}
One implication is clear by the definition. For the other, suppose $F$ is gr-flat. Then, Theorem \ref{lazard} shows that the gr-flat module $F$ is the direct limit of a direct system of finitely generated gr-free modules (in the category of $\Gamma$-graded left $R$-modules) and hence, forgetting the grading, $F$ is flat by \cite[Corollary 8.11]{so00}.
\end{proof}

\begin{remark}\label{directflat}{\em
(1) As in the ungraded case, it can be seen that every graded left $R$-module is a direct limit of all finitely generated graded submodules. Hence, using Lemma \ref{flat}, a graded left module is (gr-)flat if all of its finitely generated graded submodules are gr-flat.

(2) It follows from \cite[Theorem 3.56]{ro09}, Proposition \ref{proj} and Corollary \ref{lazard2}, that a finitely presented graded left $R$-module is gr-flat if and only if it is gr-projective. In particular, if $R$ is graded left Noetherian, then a finitely generated graded left $R$-module is gr-flat if and only if it is gr-projective.
}
\end{remark}

\section{Graded hereditary rings}

Let $\Gamma$ be a cancelation monoid with neutral element $\varepsilon$ and $R=\bigoplus_{\alpha \in \Gamma}R_{\alpha}$ be a $\Gamma$-graded ring. Graded-Dedekind domains were introduced in \cite{no82}, in the case where $\Gamma=\mathbb{Z}$ and in \cite{ac16}, in the case where $\Gamma$ is a torsionless commutative cancelation monoid. Recall that a commutative graded integral domain $R$ is called a \emph{graded-Dedekind} domain if each nonzero homogeneous ideal of $R$ is invertible. We remind that an ideal $I$ in a commutative integral domain $D$ with quotient field $K$ is called invertible if $II^{-1}=D$, where $I^{-1}=\{x\in K\mid xI\subseteq D\}$. Notice also that a nonzero ideal of a commutative integral domain is invertible if and only if it is projective \cite[Proposition 4.21]{ro09}. Hence, $R$ is a graded-Dedekind domain if and only if each nonzero homogeneous ideal of $R$ is (gr-)projective. In this section, we consider this property in a general $\Gamma$-graded ring.

\begin{definition} Assume that $\Gamma$ is a cancelation monoid. A $\Gamma$-graded ring $R=\bigoplus_{\alpha \in \Gamma}R_{\alpha}$ is graded left (resp. right) hereditary if every left (resp. right) homogeneous ideal is gr-projective.
\end{definition}

It is clear from the definition that the graded hereditary rings among commutative graded integral domains are precisely the graded-Dedekind rings.

\begin{example}\label{ex}{\em
\begin{enumerate}
  \item We can give a graduation to Small's example, which provides us a graded right hereditary ring that is not a graded left hereditary ring. Indeed, the \emph{triangular ring}
$$T:=\left(\begin{array}{cc}
                        \mathbb{Z} & \mathbb{Q} \\
                        0 & \mathbb{Q}
\end{array}\right)$$ is an $\mathbb{N}_0$-graded ring, where  $\mathbb{N}_0$ is the additive monoid of nonnegative integers, with
$T_0=\left(\begin{array}{cc}
                        \mathbb{Z} & 0 \\
                        0 & \mathbb{Q}
\end{array}\right)$,
$T_1=\left(\begin{array}{cc}
                        0 & \mathbb{Q} \\
                        0 & 0
\end{array}\right)$
and $T_i=0$ for $i>1$. According to \cite[Page 46]{la98}, all right ideals of $T$ are projective; it follows that $T$ is graded right hereditary. But the left ideal
$\left(\begin{array}{cc}
                        0 & \mathbb{Q} \\
                        0 & 0
\end{array}\right)$, which is homogeneous, is not projective. This means that $T$ is not a graded left hereditary ring.
  \item Let $k$ be a field and $R$ be a free algebra generated over $k$ by $\{x_i\mid i\in I\}$. By \cite[Page 45]{la98}, $R$ is a (left and right) hereditary ring. Let now $\{n_i\mid i\in I\}$ be a family of elements of $\mathbb{Z}$. Assigning the degree $n_i$ to $x_i$, $R$ becomes a $\mathbb{Z}$-graded ring. Then $R$ is a graded (left and right) hereditary ring.
       \item By (2), $k\langle x, y\rangle$, the free algebra generated over $k$ by $x,y$, is a $\mathbb{Z}$-graded ring, where $\deg(x)=1$ and $\deg(y)=-1$. Assume further $k$ is of characteristic zero. Then the first Weyl algebra $$A_1(k):=k\langle x, y\rangle/(xy-yx-1)$$ is a $\mathbb{Z}$-graded ring. By \cite[Page 45]{la98}, $A_1(k)$ is a (left and right) hereditary ring. Hence $A_1(k)$ is a graded (left and right) hereditary ring.
      \item Let $D$ be a Dedekind domain and $\{X_{\alpha}\}$ be a nonempty set of indeterminates over $D$. If we set $\Gamma:=\bigoplus_{\alpha}\mathbb{Z}_{\alpha}$, where each $\mathbb{Z}_{\alpha}$ is the additive group of integers, then $R=D[\{X_{\alpha},X_{\alpha}^{-1}\}]$ is a $\Gamma$-graded graded-Dedekind domain. However, if $|\{X_{\alpha}\}|\geq2$, then $R$ is not a Dedekind domain (see \cite{ac16}). Therefore a graded left (resp. right) hereditary ring need not be a left (resp. right) hereditary ring.
\end{enumerate}}
\end{example}

Let us first prove the following major result on graded submodules of gr-free left modules over graded left hereditary rings.

\begin{proposition}(Graded version of Kaplansky's Theorem)\label{kaplansky}
Assume that $\Gamma$ is a cancelation monoid, and $R=\bigoplus_{\alpha \in \Gamma}R_{\alpha}$ is a graded left hereditary ring. Then every graded submodule $M$ of a gr-free left $R$-module $F$ with the set of homogeneous basis $X$ is isomorphic to $\bigoplus_{e\in X}I_ee$ for some homogeneous left ideals $I_e$ of $R$.
\end{proposition}
\begin{proof}
Let $X=\{e_i:i\in J\}$ be a basis of homogeneous elements of $F$, and suppose the index set $J$ is well ordered. Set $F_0=0$, where 0 is the smallest index in $J$. For each $j\in J$, with $j>0$, define $$F_j:=\bigoplus_{i<j}Re_i,$$  and $\overline{F}_j=\bigoplus_{i\leq j}Re_i=F_j\oplus Re_j$. Suppose now, $M$ is a graded submodule of $F$. Assume that $\varphi_j$ is the restriction of the natural homogeneous projection $\overline{F}_j\to Re_{j}$ to the graded submodule $M\cap \overline{F}_j$. Then there is an exact sequence of graded $R$-modules and homogeneous homomorphisms $$0\to M\cap F_j\to M\cap \overline{F}_j\to \im(\varphi_j)\to0.$$ Note that $\im(\varphi_j)=I_je_j$ for some homogeneous left ideal $I_j$ of $R$. By assumption, $I_j$ is gr-projective and so $\im(\varphi_j)$ is gr-projective. It follows that the above sequence splits. Then $M\cap \overline{F}_j=(M\cap F_j)\bigoplus N_j$, where $N_j$ is isomorphic to $I_je_j$ as graded $R$-modules. One can see that $M=\bigoplus_{j\in J}N_j$ (see proof of \cite[Theorem 4.13]{ro09}) .
\end{proof}

\begin{corollary}\label{ci}
Assume that $\Gamma$ is a cancelation monoid, and $R=\bigoplus_{\alpha \in \Gamma}R_{\alpha}$ is a graded left hereditary ring. Then every graded submodule of a gr-projective left $R$-module is gr-projective.
\end{corollary}
\begin{proof}
Assume that $M$ is a graded submodule of a gr-projective left $R$-module. Hence it is a graded submodule of a gr-free left $R$-module by part $(2)\Leftrightarrow(5)$ of Proposition \ref{proj}. By Theorem \ref{kaplansky} and the assumption, we obtain that $M$ is gr-projective.
\end{proof}

Recall that a commutative graded integral domain $R=\bigoplus_{\alpha \in \Gamma}R_{\alpha}$ is a \emph{graded-PID} if each homogeneous ideal of $R$ is principal. Obviously, graded-PIDs are graded-Dedekind domains. As in Example \ref{ex}(4), let $D$ be a PID and $\{X_{\alpha}\}$ be a nonempty set of indeterminates over $D$. Then $D[\{X_{\alpha},X_{\alpha}^{-1}\}]$ is a $\bigoplus_{\alpha}\mathbb{Z}_{\alpha}$-graded-PID, but it is not a PID if $|\{X_{\alpha}\}|\geq2$ (see \cite{ac16}).

\begin{corollary}\label{}
Assume that $\Gamma$ is a cancelation monoid and $R=\bigoplus_{\alpha \in \Gamma}R_{\alpha}$ is a graded-PID. Then:
\begin{enumerate}
  \item Every graded submodule of a gr-free left $R$-module is gr-free.
  \item Every gr-projective left $R$-module is gr-free.
\end{enumerate}
\end{corollary}
\begin{proof}
(1) Assume that $M$ is a graded submodule of a gr-free left $R$-module $F$ with the set of homogeneous basis $X$. Then, by Theorem \ref{kaplansky}, $M$ is isomorphic to $\bigoplus_{e\in X}I_ee$ for some homogeneous left ideals $I_e$ of $R$. By assumption, each ideal $I_e$ is 0 or $I_e=Ra_e$ for some homogeneous nonzero element $a_e\in R$. In the case that $I_e\neq0$, assume that $a_ee$ is of degree $\alpha$, and that $F=Re'$ is a gr-free $R$-module with homogeneous basis $e'$ of degree $\alpha$. It is clear that $I_ee$ is isomorphic to $F$ as graded $R$-modules. Hence $I_ee$ is gr-free. Therefore $M$ is gr-free.

(2) follows from (1), since every gr-projective left module is a graded direct summand of a gr-free left $R$-module by Proposition \ref{proj}.
\end{proof}

The proof of the following lemma is the same as the ungraded case, so we omit it (see \cite[Page 12]{ce56}).

\begin{lemma}\label{4.18}
Assume that $\Gamma$ is a cancelation monoid and $R=\bigoplus_{\alpha \in \Gamma}R_{\alpha}$ is a graded ring. Then a graded left $R$-module $P$ is gr-projective if and only if  for any diagram of homogeneous homomorphisms of graded left $R$-modules with $Q$ gr-injective
\begin{displaymath}
\xymatrix{  &
P \ar[d]^{f} \ar@{-->}[ld]_{ h} \\
Q \ar[r]_{g} & N \ar[r] &0. }
\end{displaymath}
there is a homogeneous homomorphism $h:P\to Q$ with $gh=f$. The dual is also true.
\end{lemma}

\begin{theorem}(Graded version of Cartan-Eilenberg's Theorem)\label{cartan}
Assume that $\Gamma$ is a cancelation monoid. The following statements are equivalent for a $\Gamma$-graded ring $R=\bigoplus_{\alpha \in \Gamma}R_{\alpha}$:
\begin{enumerate}
  \item $R$ is graded left hereditary.
  \item Every graded submodule of a gr-projective left $R$-module is gr-projective.
  \item Every quotient of a gr-injective left $R$-module by a graded submodule is gr-injective.
\end{enumerate}
\end{theorem}
\begin{proof}
(1)$\Rightarrow$(2) Follows from Corollary \ref{ci}.

(2)$\Rightarrow$(1) Note that $R$ is gr-free $R$-module with homogeneous basis $\{1\}$, and hence is a gr-projective $R$-module.

(3)$\Leftrightarrow$(2) Follows as in the ungraded case, see the proof of \cite[Theorem 4.19]{ro09} using Lemma \ref{4.18} and its dual, instead of \cite[Lemma 4.18]{ro09}.
\end{proof}

The next result characterizes graded left hereditary rings in terms of flatness in the case of graded left Noetherian rings.

\begin{proposition}\label{}
Assume that $\Gamma$ is a cancelation monoid and $R=\bigoplus_{\alpha \in \Gamma}R_{\alpha}$ is a graded left Noetherian ring. Then $R$ is graded left hereditary if and only if every homogeneous left ideal is flat.
\end{proposition}
\begin{proof}
This follows from Remark \ref{directflat}(2).
\end{proof}

Let $\Gamma$ be a cancelation monoid, and $R =\bigoplus_{\alpha \in \Gamma}R_{\alpha}$ be a commutative $\Gamma$-graded integral domain. Set $\supp(R)=\{\alpha\in\Gamma\mid R_{\alpha}\neq0\}$. Hence, $\supp(R)$ is a commutative submonoid of $\Gamma$ and $R$ is a $\supp(R)$-graded ring. Then, it is harmless to assume that $\Gamma$ is a commutative cancelation monoid.

Let $\Gamma$ be a (nontrivial) commutative cancelation monoid and $\langle \Gamma \rangle = \{\alpha/\beta \mid \alpha,\beta \in \Gamma\}$ be the
quotient group of $\Gamma$. Let $R=\bigoplus_{\alpha \in \Gamma}R_{\alpha}$ be a commutative graded integral domain and $H:=\bigcup_{\alpha \in \Gamma}(R_{\alpha} \setminus \{0\})$; so $H$ is the saturated multiplicative set of nonzero homogeneous elements of $R$. Then, the localization $R_H=\bigoplus_{\gamma \in \langle \Gamma \rangle}(R_H)_{\gamma}$ is a commutative $\langle \Gamma \rangle$-graded integral domain given by $(R_H)_{\gamma}=\{r/s\mid r,s\in H, \gamma=\deg(r)/\deg(s)\}$, $\gamma\in\langle \Gamma \rangle$.

The special case of the following theorem is \cite[Theorem II.2.1]{no82} for $\Gamma=\mathbb{Z}$.

\begin{theorem}\label{}
Assume that $\Gamma$ is a commutative cancelation monoid. A commutative graded integral domain $R$ is a graded-Dedekind ring if and only if every gr-divisible left $R$-module is gr-injective
\end{theorem}
\begin{proof}
Assume that every gr-divisible $R$-module is gr-injective and that $E$ is a gr-injective left $R$-module. Then, $E$ is gr-divisible by Lemma \ref{injdiv}. Since every quotient of $E$ by a graded submodule is a gr-divisible $R$-module, hence it is gr-injective by the assumption. Whence by Theorem \ref{cartan}, $R$ is a graded hereditary integral domain, i.e., $R$ is a graded-Dedekind domain.

Conversely, assume that $R$ is a graded-Dedekind domain and $E$ is a gr-divisible $R$-module. Then by Theorem \ref{baer}, it suffices to complete the diagram
\begin{displaymath}
\xymatrix{
0\ar[r]  & I \ar@{^{(}->}[r]\ar[d]_{f} & R \ar@{-->}[dl]^{}\\
& E, &}
\end{displaymath}
where $I$ is a homogeneous ideal of $R$. Then, $I$ is invertible, and by \cite[Proposition 2.5]{aa82}, there are $a_1,\ldots,a_n\in I$ and $q_1,\ldots,q_n\in R_H$ such that $q_iI\subseteq R$ and $1=\sum_{i}q_ia_i$. Let $q_i=r_i/s$ for $r_i\in R$ and $s\in H$, $i=1,\ldots,n$. Then $r_iI\subseteq sR$ and $s=\sum_{i}r_ia_i$. Decomposing $r_i$ and $a_i$ into homogeneous components, we deduce that $r'_iI\subseteq sR$ and $s=\sum_{i}r'_ia'_i$ for some $r'_i\in H$ and $a'_i\in H\cap I$. Now letting $q'_i:=r'_i/s$, we have $q'_iI\subseteq R$ and $1=\sum_{i}q'_ia'_i$. Hence we can assume that $a_i$ and $q_i$s are homogeneous elements. Since $E$ is gr-divisible, there are homogeneous elements $e_i\in E$ such that $f(a_i)=a_ie_i$. Now, for $b\in I$, we have $$f(b)=f(\sum_iq_ia_ib)=\sum_i(q_ib)f(a_i)=\sum_i(q_ib)a_ie_i=b\sum_i(q_ia_i)e_i.$$ Hence if $e:=\sum_i(q_ia_i)e_i$, then $e\in E$ and $f(b)=be$ for all $b\in I$. Now define $h:R\to E$ by $h(r)=re$, which extends $f$ and $E$ is gr-injective.
\end{proof}

\section{Graded semihereditary rings}

Let $\Gamma$ be a cancelation monoid with neutral element $\varepsilon$ and $R=\bigoplus_{\alpha \in \Gamma}R_{\alpha}$ be a $\Gamma$-graded ring. Recall that a commutative $\Gamma$-graded integral domain $R$ is a \emph{graded-Pr\"{u}fer domain} if each nonzero finitely generated homogeneous ideal of $R$ is invertible, equivalently (gr-)projective.

\begin{definition} A graded ring $R$ is graded left (resp. right) semihereditary if every finitely generated left (resp. right) homogeneous ideal is gr-projective.
\end{definition}

It is clear from the definition that the graded semihereditary rings among commutative graded integral domains are precisely the graded-Pr\"{u}fer domains.

\begin{example}\label{}{\em
\begin{enumerate}
\item Obviously, every graded left hereditary ring is graded left semihereditary.
\item Small's example (see Example \ref{ex}(1)) is graded left semihereditary (see \cite[Page 46]{la98}).
\item We can give a graduation to Chase's example, which provides us a graded left semihereditary, but not a graded right semihereditary ring. Let $S$ be a von Neumann regular and nonsemisimple ring. Hence there is an ideal $I$ of $S$ such that, as a submodule of the right $S$-modules $S$, $I$ is not a direct summand. Let $R:=S/I$, which is also a von Neumann regular ring. The triangular ring
$$T:=\left(\begin{array}{cc}
                        R & R \\
                        0 & S
\end{array}\right)$$ is an $\mathbb{N}_0$-graded ring, with
$T_0=\left(\begin{array}{cc}
                        R & 0 \\
                        0 & S
\end{array}\right)$,
$T_1=\left(\begin{array}{cc}
                        0 & R \\
                        0 & 0
\end{array}\right)$
and $T_i=0$ for $i>1$. According to \cite[Page 47]{la98}, $T$ is left semihereditary; it follows that $T$ is graded left semihereditary. But the right ideal
$\left(\begin{array}{cc}
                        0 & R \\
                        0 & 0
\end{array}\right)$, which is homogeneous, is not projective. This means that $T$ is not a graded right semihereditary ring.
\item Let $D$ be a Pr\"{u}fer domain which is not a field and set $R:=D[X,X^{-1}]$, for an indeterminate $X$ over $D$. Then by \cite[Example 3.6]{ac13}, $R$ is a $\mathbb{Z}$-graded-Pr\"{u}fer domain which is not Pr\"{u}fer. Therefore a graded (left, right) semihereditary ring need not be a (left, right) semihereditary ring.
\end{enumerate}}
\end{example}

\begin{proposition}\label{direct}
Assume that $\Gamma$ is a cancelation monoid, and $R=\bigoplus_{\alpha \in \Gamma}R_{\alpha}$ is a graded left semihereditary ring. Then every finitely generated graded submodule $M$ of a gr-free left $R$-module $F$ with the set of homogeneous basis $X$ is isomorphic to $I_1e_1+\cdots+I_ne_n$ for some finitely generated homogeneous left ideals $I_1,\ldots,I_n$ of $R$ and $e_1,\ldots,e_n\in X$.
\end{proposition}
\begin{proof}
Let $X=\{e_i:i\in I\}$ be a basis of homogeneous elements of $F$, and let $M=\langle x_1,\ldots,x_m\rangle$ be a graded finitely generated submodule of $F$. Each $x_k$ is a finite linear combination of $e_i$'s. Collect all of these $e_i$'s as a set $Y$ which is finite and $M\subseteq\langle Y\rangle$. Hence $\langle Y\rangle$ is a gr-free submodule of $F$, and so we may assume that $F$ is finitely generated with homogeneous basis $\{e_1,\ldots,e_n\}$, and $\alpha_i:=\deg(e_i)$.

We prove, by induction on $n\geq1$, that $M$ is isomorphic to $I_1e_1+\cdots+I_ne_n$ for some finitely generated homogeneous left ideals $I_1,\ldots,I_n$ of $R$. If $n=1$, then there exists a finitely generated homogeneous left ideal $I$ of $R$ such that $M=Ie_1$. Assume that $n>1$ and that the assertion holds for $n-1$. Now by the notation of the proof of Theorem \ref{kaplansky}, $M=(M\cap F_n)\oplus N_n$, where $F_n=\oplus_{i=1}^{n-1}Re_i$ and $N_n\cong I_ne_n$ for some homogeneous left ideal $I_n$ of $R$. Being a direct summand of $M$, $M\cap F_n$ is a finitely generated graded submodule of $F_n$. Hence, the induction hypothesis completes the proof.
\end{proof}

\begin{theorem}\label{Albrecht}
Assume that $\Gamma$ is a cancelation monoid. Then a graded ring $R=\bigoplus_{\alpha \in \Gamma}R_{\alpha}$ is graded left semihereditary if and only if every finitely generated graded submodule $M$ of a gr-projective left $R$-module $P$ is gr-projective.
\end{theorem}
\begin{proof}
Assume that $R$ is graded left semihereditary and that $M$ a finitely generated graded submodule of a gr-projective left $R$-module $P$. Since $P$ is gr-projective, it is a graded submodule, even a summand, of a gr-free left $R$-module $F$. Hence $M$ is a finitely generated graded submodule of $F$. By Proposition \ref{direct}, $M$ is isomorphic to $I_1e_1+\cdots+I_ne_n$ for some finitely generated homogeneous left ideals $I_1,\ldots,I_n$ of $R$ and $e_1,\ldots,e_n$ elements of a homogeneous basis of $F$. As each of these $I_i$s is gr-projective, using Proposition \ref{proj}, $M$ is gr-projective too. Conversely, every finitely generated homogeneous left ideal is a finitely generated graded submodule of the gr-free $R$-module $R$. Hence, such ideals are gr-projective, and $R$ is graded left semihereditary.
\end{proof}

A graded ring $R$ is called a \emph{graded left coherent ring} if every finitely generated graded left
ideal is finitely presented. Using \cite[Proposition 3.11]{ro09}, every left semihereditary ring is left coherent.

\begin{proposition}\label{coh-semi}
Assume that $\Gamma$ is a cancelation monoid and that $R=\bigoplus_{\alpha \in \Gamma}R_{\alpha}$ is a $\Gamma$-graded ring. Then $R$ is graded left semihereditary if and only if $R$ is graded left coherent and every homogeneous left ideal is gr-flat.
\end{proposition}
\begin{proof}
Assume first that $R$ is graded left semihereditary. Then every finitely generated homogeneous left ideal $I$ is (gr-)projective; hence, $I$ is finitely presented, by \cite[Proposition 3.11]{ro09}. Therefore $R$ is graded left coherent. Since every finitely generated homogeneous left ideal is (gr-)projective, it is (gr-)flat. It follows from Remark \ref{directflat}(1) that every
graded left ideal is gr-flat.

Conversely, if $I$ is a finitely generated homogeneous left ideal, then $I$ is a graded submodule of the gr-flat module $R$, and so $I$ is gr-flat; since $R$ is left coherent, $I$ is also finitely presented. Hence, $I$ is projective, by Remark \ref{directflat}(2), and so $R$ is graded left semihereditary.
\end{proof}

Recall that the monoid $\Gamma$ is \emph{torsionless} if from $\gamma^n=\gamma'^n$, where $n\geq1$ is an integer and $\gamma,\gamma'\in\Gamma$, follows $\gamma=\gamma'$. It is well known that the cancellation monoid $\Gamma$ is torsionless if and only if it can be given a total order $\leq$ compatible with the monoid operation \cite[page 123]{no68}.

From now on we assume that $\Gamma$ is a torsionless commutative cancellation monoid. Let $R=\bigoplus_{\alpha \in \Gamma}R_{\alpha}$ be a commutative $\Gamma$-graded integral domain and $M$ be a $\Gamma$-graded left $R$-module. The \emph{torsion submodule of $M$} is defined as
$$T(M):=\{x\in M\mid rx=0,\text{ for some }0\neq r\in R\}.$$ Since $R$ is a commutative integral domain, $T(M)$ is a submodule of $M$. The $R$-module $M$ is called \emph{torsion-free} if $T(M)=0$. In the following lemma we show that $T(M)$ is a graded submodule.

\begin{lemma}
Assume that $\Gamma$ is a torsionless commutative cancelation monoid, $R=\bigoplus_{\alpha \in \Gamma}R_{\alpha}$ is a commutative $\Gamma$-graded integral domain and $M$ is a $\Gamma$-graded left $R$-module. Then $T(M)$ is a graded submodule of $M$.
\end{lemma}
\begin{proof}
Assume that $x$ is a nonzero element of $T(M)$, and that $x=x_{\beta_1}+\cdots+x_{\beta_m}$ is the decomposition of $x$ into nonzero homogeneous components with $\beta_1<\cdots<\beta_m$. There exists a nonzero element $r\in R$ such that $rx=0$. Assume also that $r=r_{\alpha_1}+\cdots+r_{\alpha_n}$ is the decomposition of $r$ into nonzero homogeneous components with $\alpha_1<\cdots<\alpha_n$. It follows from $rx=0$ that $r_{\alpha_1}x_{\beta_1}=0$. Hence, $x_{\beta_1}\in T(M)$ and $x-x_{\beta_1}=x_{\beta_2}+\cdots+x_{\beta_m}\in T(M)$. By the same way, we see that $x_{\beta_2},\ldots,x_{\beta_m}\in T(M)$. This completes the proof.
\end{proof}

Let $\Gamma$ be a (nontrivial) commutative cancelation monoid and $\langle \Gamma \rangle = \{\alpha/\beta \mid \alpha,\beta \in \Gamma\}$ be the
quotient group of $\Gamma$. Let $R=\bigoplus_{\alpha \in \Gamma}R_{\alpha}$ be a commutative $\Gamma$-graded integral domain and $H = \bigcup_{\alpha \in \Gamma}(R_{\alpha} \setminus \{0\})$. We observed that the localization $R_H$ is a commutative $\langle \Gamma \rangle$-graded integral domain. Notice that each nonzero homogeneous elements of $R_H$ is invertible.

\begin{lemma}\label{imbed}
Assume that $\Gamma$ is a torsionless commutative cancelation monoid, $R=\bigoplus_{\alpha \in \Gamma}R_{\alpha}$ is a commutative graded integral domain, and $M$ is a graded torsion-free left $R$-module. Then $M$ can be embedded in a gr-free left $R_H$-module. Moreover, if $M$ is finitely generated, then $M$ can be embedded in a finitely generated gr-free left $R$-module.
\end{lemma}
\begin{proof}
Imbed $M$ in a gr-injective left module $E$ by Theorem \ref{injectivehull}. Since $M$ is torsion-free, $M$ is also imbedded in $E/T(E)$, which is torsion-free and gr-divisible. It is easy to see that $E/T(E)$ is a graded left $R_H$-module, hence, is a gr-free $R_H$-module by \cite[Lemma 4]{cs18}.

Assume $M$ is finitely generated and choose a homogeneous basis for the gr-free $R_H$-module $E/T(E)$. Then, each of the homogeneous generators of $M$ is a linear combination of only finitely many elements of the basis. Thus, we may assume $M$ is imbedded in a finitely generated gr-free $R_H$-module $V$. Let $\{v_1,\ldots,v_m\}$ be a homogeneous basis of $V$ and $\{x_1,\ldots,x_n\}$ be a generating set of homogeneous elements of $M$. Then each $x_i=\sum_{j}(r_{ij}/s_{ij})v_j$, where $r_{ij}\in R$ and $s_{ij}\in H$. If $s$ is the product of all $s_{ij}$, then $\{s^{-1}v_1,\ldots,s^{-1}v_m\}$, is independent in $V$, and the graded $R$-submodule $N$ of $V$ generated by this set is gr-free. Clearly $M$ is imbedded in $N$.
\end{proof}

In the following result, we give a new characterization of graded-Pr\"{u}fer domains in terms of modules.

\begin{theorem}\label{fg}
Assume that $\Gamma$ is a torsionless commutative cancelation monoid. A commutative graded integral domain $R=\bigoplus_{\alpha \in \Gamma}R_{\alpha}$ is a graded-Pr\"{u}fer domain if and only if every finitely generated torsion-free graded $R$-module is gr-projective.
\end{theorem}
\begin{proof}
Assume that $R$ is a graded-Pr\"{u}fer domain and $M$ is a finitely generated torsion-free graded $R$-module. Then by Lemma \ref{imbed}, $M$ can be imbedded in a finitely generated gr-free $R$-module. It follows from Theorem \ref{Albrecht} that $M$ is gr-projective. Conversely, every homogeneous ideal of $R$ is torsion-free, so every finitely generated homogeneous ideal of $R$ is gr-projective.
\end{proof}

\begin{corollary}\label{f}
Assume that $\Gamma$ is a torsionless commutative cancelation monoid, and that $R=\bigoplus_{\alpha \in \Gamma}R_{\alpha}$ is a graded-Pr\"{u}fer domain. Then every graded $R$-module is gr-flat if and only if it is torsion-free.
\end{corollary}
\begin{proof}
Suppose that $F$ is gr-flat. Thus $F$ is flat by Corollary \ref{lazard2}, and so is torsion-free by \cite[Proposition 3.49]{ro09}. Conversely, assume that $F$ is a torsion-free graded $R$-module. By Remark \ref{directflat}, it is enough to show that every finitely generated graded submodule of $F$ is gr-flat. Since $R$ is a graded-Pr\"{u}fer domain, by Theorem \ref{fg}, every such finitely generated graded submodule is gr-projective, hence is gr-flat.
\end{proof}

Let $\{M_{j}\}_{j\in J}$ be a family of $\Gamma$-graded left $R$-modules. The direct product $^*\Pi_{j}M_j$ exists in the category of $\Gamma$-graded left $R$-modules with $(^*\Pi_{j}M_j)_{\alpha}=\Pi_{j}(M_j)_{\alpha}$ for all $\alpha\in\Gamma$. It is clear that $^*\Pi_{j}M_j=\Pi_{j}M_j$ whenever $\Gamma$ is a finite group. In the following proposition, we do not assume that $\Gamma$ is torsionless.

\begin{proposition}
Assume that $\Gamma$ is a finite group, $R=\bigoplus_{\alpha \in \Gamma}R_{\alpha}$ is a commutative $\Gamma$-graded integral domain and that every torsion-free graded $R$-module is gr-flat. Then $R$ is a graded-Pr\"{u}fer domain.
\end{proposition}
\begin{proof}
Let $J$ be an arbitrary set. Since $\Gamma$ is finite, the direct product $R^J:=\Pi_{j\in J}R$ is a $\Gamma$-graded $R$-module by the paragraph above. Since $R^J$ is torsion-free, by the hypothesis, $R^J$ is gr-flat, hence is flat. It follows from \cite[Theorem 4.47]{la98} that $R$ is coherent. Let now $I$ be a homogeneous ideal of $R$. Then $I$ is torsion-free and thus is gr-flat by the hypothesis. Therefore $R$ is a graded-Pr\"{u}fer domain by Proposition \ref{coh-semi}.
\end{proof}

\noindent {\bf Acknowledgement.}
The authors would like to thank the referee for careful reading of the manuscript that improved the presentation of the paper.

\end{document}